\documentclass[12pt]{amsart}
\usepackage{amsfonts}
\usepackage{amssymb}
\usepackage[letterpaper, left=2.5cm, right=2.5cm, top=2.5cm,
bottom=2.5cm,dvips]{geometry}
\usepackage{verbatim}
\usepackage[T1]{fontenc}
\usepackage{graphicx}
\usepackage{amsmath}

\usepackage{tikz}

\newtheorem{theorem}{Theorem}
\theoremstyle{plain}

\newtheorem{algorithm}{Algorithm}

\newtheorem{claim}{Claim}

\newtheorem{conjecture}{Conjecture}

\numberwithin{equation}{section}

\usepackage{array,multirow}
\usepackage{algorithm}
\usepackage{algpseudocode}

\usepackage{xcolor}

\usepackage[backref]{hyperref}
\hypersetup{
	colorlinks,
	linkcolor={blue!60!black},
	citecolor={red!60!black},
	urlcolor={red!60!black}
}

\newcommand{\oeis}[1]{\href{https://oeis.org/#1}{\ensuremath{\mathtt{#1}}}}

\begin{document}
	\title{Finite and infinite barrycades}
	
	\author[M. D\k{e}bsk]{Micha\l{} D\k{e}bski}
	\address{Faculty of Mathematics and Information Science, Warsaw University
		of Technology, 00-662 Warsaw, Poland}
	\email{m.debski@mini.pw.edu.pl}
	
	\author[J. Grytczuk]{Jaros\l aw Grytczuk}
	\address{Faculty of Mathematics and Information Science, Warsaw University
		of Technology, 00-662 Warsaw, Poland}
	\email{j.grytczuk@mini.pw.edu.pl}

	\author[P. Naroski]{Pawe\l{} Naroski}
	\address{Faculty of Mathematics and Information Science, Warsaw University
		of Technology, 00-662 Warsaw, Poland}
	\email{p.naroski@mini.pw.edu.pl}
	
	\author[B.Pawlik]{Bart\l omiej Pawlik}
	\address{Institute of Mathematics, Silesian University of Technology, 44-100 Gliwice, Poland}
	\email{bpawlik@polsl.pl}
	
	\author[J. Przyby\l o]{Jakub Przyby\l o}
	\address{AGH University of Krakow, al. A. Mickiewicza 30, 30-059 Krakow, Poland}
	\email{jakubprz@agh.edu.pl}
	
	\author[M. Śleszy\'nska-Nowak]{Ma\l gorzata Śleszy\'nska-Nowak}
	\address{Faculty of Mathematics and Information Science, Warsaw University
		of Technology, 00-662 Warsaw, Poland}
	\email{malgorzata.nowak@pw.edu.pl}

\begin{abstract}

An $n$-barrycade of height $h$ is a set of $h$ permutations of $[n]$ such that all the prefix sums are different. This notion was described in 2020 by Richard Guy, together with the central problem to determine for which values of $n$ there exists a break-free $n$-barrycade, that is, one in which every possible prefix sum from $1$ to $\frac{n(n+1)}{2}-1$ is covered. The height of such barrycade would be $\frac{n+2}{2}$.

We prove that barrycades of linear height exist for infinitely many sizes: there is a constant $c>0$ such that for infinitely many positive integers $n$ there exists an $n$-barrycade of height at least $cn$. We also explore an infinite variant of the problem and propose three constructions -- greedy, grasshopper and precise grasshopper -- that give increasingly more satisfying results. Along the way we encounter new integer sequences and formulate conjectures concerning omitted elements, missing partial sums, and word representations of infinite barrycades, which are firmly supported by computational data.
\end{abstract}

\maketitle

\section{Introduction}

Let $[n]=\{1,2,\ldots, n\}$. For any finite sequence $a=(a_1,a_2,\ldots,a_n)$ of integers we define a~\emph{proper partial sum} $$s_a(i)=a_1+a_2+\ldots+a_i$$ for every $i\in[n-1]$. The set of all proper partial sums of $a$ is denoted by $S_a$. An \emph{$n$-barrycade} is a collection $(\pi_1,\pi_2,\ldots,\pi_h)$ of permutations $\pi_i$ of the set $[n]$ such that $S_{\pi_i}\cap S_{\pi_j}=\emptyset$ for every $i,j\in[h]$. For example the collection $$\big((1,2,3,4),\,(2,3,4,1),\,(4,3,1,2)\big)$$ is a $4$-barrycade, while $$\big((1,2,3,4),\,(2,1,4,3)\big)$$ is not, since  $$S_{(1,2,3,4)}\cap S_{(2,1,4,3)}=\{3\}.$$
For an $n$-barrycade $B=(\pi_1,\pi_2,\ldots,\pi_h)$, we call $n$ its \emph{size}, $h$ its \emph{height} and $\pi_i$ its $i$-th \emph{row}.

A $n$-barrycade $(\pi_1,\pi_2,\ldots,\pi_h)$ is \emph{break-free} if every possible proper partial sum is covered by a permutation in the barrycade. If this is the case, then we necessarily have  $$\bigcup_{i=1}^hS_{\pi_i}=[N],$$ where $N=(1+2+\cdots+ n)-1=\frac{n(n+1)}{2}-1=\frac{(n-1)(n+2)}{2}.$ Moreover, since each set $S_{\pi_i}$ has exactly $n-1$ elements, a break-free $n$-barrycade of height $h$ must satisfy $h(n-1)=N$. Hence $$h=\frac{n+2}{2},$$which implies that $n$ must be even (or $n=1$).

Barrycades have a very natural graphical interpretation. Each number $L$ can be represented as a log, that is, a rectangle of length $L$ and unit height, and each permutation as a sequence of logs joined along their unit sides. A $n$-barrycade is then a rectangle of lenght $n(n+1)/2$ and height $h$ composed of such sequences of logs (rows) stacked on top of one another. Figure~\ref{Barry4} presents a~visualization of the break-free barrycade $$\displaystyle\big((1,6,4,2,5,3),\,(2,3,4,5,1,6),\,(3,1,6,4,5,2),\,(6,2,4,5,3,1)\big).$$

\begin{figure}\label{Barry4}
\centering
\begin{tikzpicture}[scale=0.75, barry/.style={line width=1.8pt}]
\draw[step=1, gray!60, dash pattern=on 2pt off 2pt] (-0.5,-0.5) grid (21.5,4.5);

\draw[barry] (0,0) rectangle (21,4);

\foreach \y in {1,2,3} {\draw[barry] (0,\y) -- (21,\y);}
\foreach \x in {1,7,11,13,18} {\draw[barry] (\x,3) -- (\x,4);}
\foreach \x in {2,5,9,15,16} {\draw[barry] (\x,2) -- (\x,3);}
\foreach \x in {3,4,10,14,19} {\draw[barry] (\x,1) -- (\x,2);}
\foreach \x in {6,8,12,17,20} {\draw[barry] (\x,0) -- (\x,1);}
\end{tikzpicture}
\caption{Barrycade $\displaystyle\big((1,6,4,2,5,3),\,(2,3,4,5,1,6),\,(3,1,6,4,5,2),\,(6,2,4,5,3,1)\big)$.}
\end{figure}
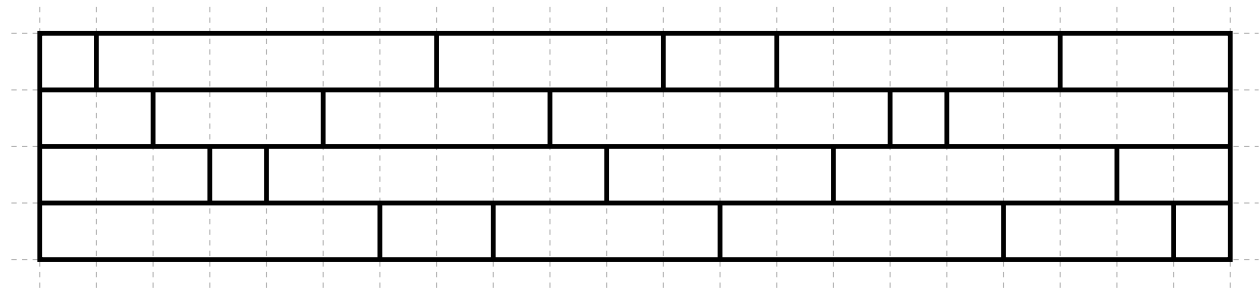

Barrycades\footnote{Guy used the spelling \emph{barrycade}, rather than \emph{barricade}, in honor of Barry Cipra, who prompted him to take up this topic. We also follow this convention.} were described by Richard Guy in 2020 \cite{Guy2020}, who was particularly interested in determining for which values of $n$ a break-free $n$-barrycade exists. He explicitly, though informally, conjectured that break-free $n$-barrycades exist for every even~$n$, writing that “it seems certain” that this is the case, while noting that no proof was known. We formulate his belief formally as follows.

\begin{conjecture}\label{Guyc}
For every even $n$ there exists a break-free $n$-barrycade.
\end{conjecture}

So far, mostly experimental results supporting the conjecture were derived. After the first version of this paper appeared, Bini\k{e}da, D\k{e}bski, Gutowski and Milewski~\cite{Bin2026} computationally confirmed the conjecture for all even $n\leq 98$. Moreover, the number of distinct break-free $n$-barrycades seems to grow very quickly. For instance, for $n=6$ there are at least $1120$ of them, while for $n=8$, there are at least $28\,432\,700$ of them (see \cite{Guy2020}.) Let us also mention that in 2011, Stan Wagon presented a version of the problem as Macalester College Problem of the Week 1148 \cite{Wag2011}. 

Our main purpose is to give a further support for Conjecture \ref{Guyc} by proving that $n$-barrycades of linear height exist (Theorem \ref{mt}). We also investigate a natural infinite variant of the problem in several versions. The paper is concluded with some comments on future research and a~reformulation of the problem in terms special words whose letters satisfy prescribed distance constraints. This interpretation may open a way for further connections, in particular, to some objects of interest in music, as suggested by Guy at the end of his article \cite{Guy2020}.

After the first version of this paper appeared, Bini\k{e}da et al.~\cite{Bin2026} obtained a substantial strengthening of our finite result. Building on some of the ideas developed here, they found a remarkably efficient and elegant construction showing that for every positive integer $h$ and every $n\geq 2h+3$ there exists an $n$-barrycade of height $h$.

In particular, for every $n\geq 5$ there exists an $n$-barrycade of height at least $\lfloor (n-3)/2\rfloor$. They also computationally verified the conjectured optimal construction for every height $2\leq h\leq 50$.

Let us also point on a wider context into which the topic of barrycades fits nicely. It includes a variety of problems involving permutations or words and having similar additive flavor. For, instance one may study permutations with maximum number of block-sums or words with forbidden additive block-patterns (see \cite{CassaigneCSS}, \cite{Konieczny}).

\section{Main result}

Conjecture \ref{Guyc} equivalently states that for every even $n$ there exists an $n$-barrycade of height $(n+2)/2$. A straightforward relaxation of this problem is to ask for a maximum possible height of an $n$-barrycade; we prove that it's at least linear in $n$. This result was subsequently substantially strengthened by Bini\k{e}da et al.~\cite{Bin2026}; see the discussion in the Introduction.

\begin{theorem}\label{mt}
There exists a constant $c>0$ such that for infinitely many positive integers $n$ there exists an $n$-barrycade of height at least $cn$.
\end{theorem}

\begin{proof}
It is enough to show that there exists an $\mathcal{O}(n)$-barrycade of height $(h-1)$ for every positive integer $h$.

We start by building an auxiliary configuration with $h$ rows, such that the $i$-th row starts with a block of length $i$, that is followed by a carefully chosen sequence $P_i$. We pick a~sufficiently large constant $C$ and then select $P_i$'s according to the following claim.

\begin{claim}
\label{claim_finiteBarrClaim}
It is possible to choose sequences $P_1, P_2, \ldots, P_h$ of pairwise distinct block lengths smaller than $22h$, such that
\begin{enumerate}
\item[a)] every block length appears at most once among $P_1, \ldots, P_h$,
\item[b)] for every $j$ from $1$ to $h-1$, block of length $j$ appears in $P_j$ and is followed by block of length $2h-j$,
\item[c)] blocks of length $h$ and $2h$ don't appear in any $P_i$,
\item[d)] placing $P_1, \ldots, P_h$ in a barrycade, such that each $P_i$ starts one unit to the right from the start of $P_{i-1}$, creates no coincident joints,
\item[e)] all $P_i$'s have the same total length.
\end{enumerate}
\end{claim}

We will prove Claim \ref{claim_finiteBarrClaim} by building all the $P_i$'s simultaneously. First, we choose an increasing sequence of block lengths $x_1, x_2, \ldots, x_{h-1}$ with $x_1>2h$ and $x_{h-1} < 11h$, then for each $i$ from $1$ to $h-1$, we start the sequence $P_i$ by $x_i$, followed by a block of length $i$ and then a block of length $2h-i$. This ensures b) and doesn't violate a) nor c), but we need to choose $x_i$'s so that d) is not violated. 

We set $x_1=2h+1$ and then, having chosen $x_1, \ldots, x_{i-1}$, we pick $x_i$ to be the smallest number $d$ larger that $x_{i-1}$ such that the three values $P_{d,i}:=\lbrace x_i+i, x_i+2i, x_i+2i+(2h-i)\rbrace$ are disjoint from $V_k:=\lbrace x_k+k, x_k+2k, x_k+2k+(2h-k)\rbrace$ for every $k<i$; this clearly ensures d). When choosing $x_i$ we start with $d=x_{i-1}+1$ and, while $P_{d,i}$ intersects with some $V_k$, we say that $x_k$ causes a \emph{skip} and increase $d$ by $1$. A crucial observation is that throughout the whole procedure, for each $k$, $x_k$ can cause at most $9$ skips. Indeed, throughout the procedure, the smallest value from $P_{d,i}$ can only increase as $d$ or $i$ increases, so it can coincide with the smallest value from $V_k$ at most once; and the same is true for any combination of smallest/medium/largest value from $P_{d,i}$ and smallest/medium/largest value from $V_k$. It follows that the procedure made less than $9h$ skips, so $x_{h-1}<2h+9h+h=11h$.

Now we will append one more term to each $P_i$ to ensure e). So far, for every $i$ we have appended to $P_i$ blocks of total length $x_i+2h$, so we add a block of length $11h+x_{h-1}-x_i$ to make the total length of $P_i$ exactly $x_{h-1}+13h$. Note that since $x_i$'s formed a strictly increasing sequence and all the added block are longer than $x_{h-1}$, we have violated neither a), c) or d). The longest of the added blocks is $11h+x_{h-1}-x_1 < 22h$.

Finally, we choose $P_h$ to be a single block of length $x_{h-1}+13h$, which is at most $24h$. Since it doesn't violate any conditions from a) to e), the proof of the claim is complete.

Now we will use sequences $P_1, \ldots, P_h$ given by Claim \ref{claim_finiteBarrClaim} to build a $24h$-barrycade of height $h$. As announced earlier, for every $i$ from $1$ to $h$, we start the $i$-th row by a block of length $i$, followed by $P_i$. By d), this doesn't create any coincident joints. Next, we append a block of length $h$ to the first row, and then append one of $P_i$'s to each row in the same order, but shifted cyclically, i.e. the block $h$ in the first row is followed by $P_{h-1}$, $P_1$ is appended to the second row, $P_2$ is appended to the third row and so on. Note that by e) the end of an added block of length $h$ comes one position after the end of a copy of $P_h$ in the last row; therefore, once again using d), we observe that int the part of a barrycade that we have build, all the proper partial sums are different. We repeat this operation $h-2$ more times, i.e. for $j$ from $2$ to $h-1$ we append a block of length $h$ to the $j$-th row, then to each row $i$ append a copy of $P_{i-j}$, where the subtraction is taken modulo $h$; again, by conditions d) and e) from the claim, the part of the barrycade we constructed has all it's proper partial sums different.

Next, we delete the last row. Note that at this point, for every $i$, the $i$-th row contain a block of length $i$, a block of length $h$, and one copy of each of the sequences $P_1, P_2, \ldots, P_h$; in particular, ends of all the rows are in $h-1$ consecutive positions. By conditions a) and c) of the claim, in each row all the lengths of used blocks are distinct, except the length $i$, that appears twice in the $i$-th row. Also, the set of used block lengths is the same for each row, and it contains all lengths from $1$ to $h-1$ by b).

Now, for every block length $\ell\leqslant 24h$ that has not been used so far, except $2h$, we append $\ell$ to every remaining row; since $\ell>h$, this doesn't create any coincident joints. Then, for every $i$, we put a block of length $2h-i$ at the end of the $i$-th for.

We finalize the construction by replacing, for every $i$, the two consecutive block of lengths $i$ and $2h-i$ in the $i$-th row by a block of length $2h$; recall that those block are consecutive by b). At this point each block length from $1$ to $24h$ appears exactly once in each row and there are no coincident joints, so we have built a $24h$-barrycade of height $h-1$, thereby completing the proof.
\end{proof}

\section{Infinite barrycades}
	
The goal of this section is to prove an infinite version of Conjecture \ref{Guyc}. We will provide a number of constructions that are increasingly more successful and eventually lead to the desired outcome, but they also leave some loose ends that sparked our curiosity. We start by giving infinite analogs of earlier definitions.

The constructions developed here, together with their associated sequences of omitted elements, missing partial sums, first row elements, and word representations, led to thirteen new entries in the On-Line Encyclopedia of Integer Sequences (OEIS): \oeis{A399897}--\oeis{A399909} \cite{OEIS}.

Let \(\mathbb{N}\) denote the set of positive integers. An infinite sequence $\mu=(a_1,a_2,a_3,\ldots)$ of positive integers, with \(a_i\neq a_j\) for \(i\neq j\), is called a \emph{quasi-permutation}. The set
$$S_\mu=\{a_1,\,\ a_1+a_2,\,\ a_1+a_2+a_3,\,\ldots\}$$
is called the \emph{set of partial sums} of the quasi-permutation $\mu=(a_1,a_2,a_3,\ldots)$. A special case of a quasi-permutation is a bijection $\mathbb{N}\to\mathbb{N}$, which is called a \emph{permutation of $\mathbb{N}$}.

A \emph{$\mathbb{N}$-quasi-barrycade} is a collection of quasi-permutations $(\mu_1,\,\mu_2,\,\mu_3,\,\ldots)$ such that $$S_{\mu_i}\cap S_{\mu_j}=\emptyset\ \mbox{ for all }\ i\neq j.$$ The quasi-permutation $\mu_i$ is called an \emph{$i$-th row} of $\mathbb{N}$-quasi-barrycade $(\mu_1,\,\mu_2,\,\mu_3,\,\ldots)$. An \emph{$\mathbb{N}$-barrycade} is an $\mathbb{N}$-quasi-barrycade that consists only of permutations of $\mathbb{N}$.

In the remainder of this section, every visualization of a barrycade will be built from the bottom up. The dashed top and right boundaries indicate that the barrycade continues upward and to the right. The pseudocodes of the corresponding algorithms, additional terms of each sequence described below, as well as the prefixes of infinite word representations of described infinite barrycades, are listed in Appendix~\ref{apA}. 

\subsection{Greedy $\mathbb{N}$-quasi-barrycade $B_\varrho$}

We start by constructing a quasi-barrycade of infinite height with a simple greedy approach.

\begin{theorem}\label{infthm1}
There exists a break-free $\mathbb{N}$-quasi-barrycade of infinite height.
\end{theorem}
\begin{proof}

We construct rows in order, starting from the first. Having chosen rows form $1$ to $i-1$ and a prefix of the $i$-th row, we choose the next term to be the smallest positive integer not yet used in the prefix whose addition does not create a~partial sum coinciding with any partial sum already occurring in any of the previous rows -- i.e. the first term in each row is the smallest positive integer that is not a partial sum of any of the previous rows, which ensures that the barrycade we build is break-free. Note that a choice of the next term is always possible, because the $m$-th partial sum in each row is at least $m(m+1)/2$, so each row has at most $O\left(\sqrt{x}\right)$ partial sums smaller than $x$; therefore, the first $i-1$ rows leave an infinite number of natural numbers that are not partial sums. It completes the proof.
\end{proof}

Theorem \ref{infthm1} does not properly solve the infinite counterpart of Conjecture \ref{Guyc}, because the rows of the constructed quasi-barrycade are not permutations of $\mathbb{N}$. However, it seems to be surprisingly close.

\begin{figure}[t]
\centering
\begin{tikzpicture}[scale=0.7,barry/.style={line width=1.8pt, line cap=butt},cont/.style={line width=1.8pt, line cap=butt, dash pattern=on 4pt off 3pt}]
\def\W{22}
\def\H{5}

\draw[step=1, gray!60, dash pattern=on 2pt off 2pt] (-0.5,-0.5) grid (\W+0.5,\H+0.5);

\draw[barry] (0,0) -- (\W+0.5,0);
\draw[barry] (0,0) -- (0,\H+0.5);

\draw[cont] (0,\H+0.5) -- (\W+0.5,\H+0.5);
\draw[cont] (\W+0.5,0) -- (\W+0.5,\H+0.5);

\foreach \y in {1,2,3,4,5} {\draw[barry] (0,\y) -- (\W+0.5,\y);}

\foreach \x in {1,3,6,10,15,21} {\draw[barry] (\x,0) -- (\x,1);}
\foreach \x in {2,5,9,14,20} {\draw[barry] (\x,1) -- (\x,2);}
\foreach \x in {4,7,8,13,19} {\draw[barry] (\x,2) -- (\x,3);}
\foreach \x in {11,12,16,18} {\draw[barry] (\x,3) -- (\x,4);}
\foreach \x in {17,22} {\draw[barry] (\x,4) -- (\x,5);}
\end{tikzpicture}
\caption{A fragment of the greedy $\mathbb{N}$-quasi-barrycade $B_\varrho$.}
\label{Gvarrho}
\end{figure}
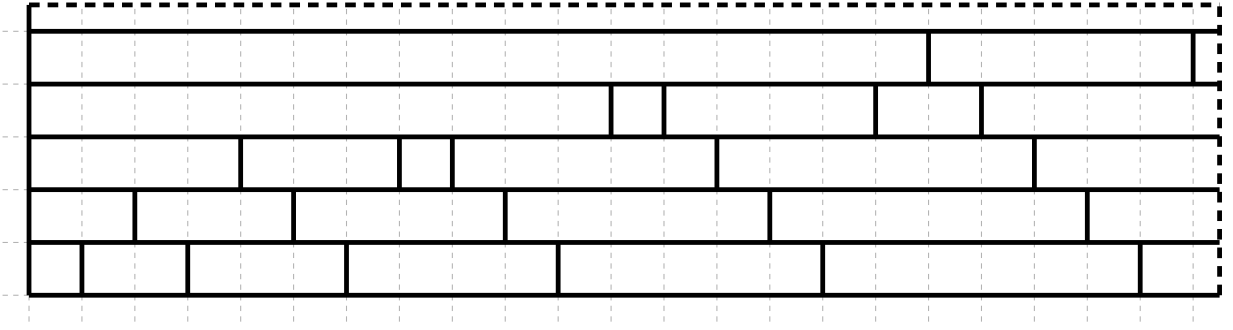

We call the barrycade constructed in the above proof the greedy $\mathbb{N}$-quasi-barrycade $B_\varrho=\{\varrho_1,\varrho_2,\varrho_3,\ldots\}$. It is easy to see that the first four rows of $B_\varrho$ are 
\begin{align*}
\varrho_1&=(1,\,2,\,3,\,4,\,5,\,6,\,7,\,8,\,9,\,10,\,11,\,12,\,13,\,14,\,\ldots),\\
\varrho_2&=(2,\,3,\,4,\,5,\,6,\,7,\,8,\,9,\,10,\,11,\,12,\,13,\,14,\,15,\,\ldots),\\
\varrho_3&=(4,\,3,\,1,\,5,\,6,\,7,\,8,\,9,\,10,\,11,\,12,\,13,\,14,\,15,\,\ldots),\\
\varrho_4&=(11,\,1,\,4,\,2,\,5,\,6,\,3,\,7,\,8,\,9,\,12,\,13,\,14,\,15,\,\ldots).
\end{align*}

The first row $\varrho_1$ is a permutation of $\mathbb{N}$. In contrast, $\varrho_2$, $\varrho_3$, and $\varrho_4$ are not permutations, but each of them omits exactly one number -- $1$, $2$, and $10$, respectively. Computational experiments suggest that this phenomenon persists; we didn't manage to prove it, but conjecture that it is indeed the case.

Let us consider a sequence~\oeis{A399907} such that for $i\geqslant{2}$ we have $\oeis{A399907}(i)$ is the smallest number that is omitted in the quasi-permutation $\varrho_i$. Therefore, the first 15 terms of \oeis{A399907} are  
$$1,\, 2,\, 10,\, 16,\, 29,\, 39,\, 51,\, 74,\, 101,\, 129,\,131,\,165,\,202,\,235,\,281.$$ 

The omitted numbers are seemingly related to first terms of the rows. Let the sequence \oeis{A399908} consist of the first elements of the rows $\varrho_i$ of $B_\varrho$, i.e., $\oeis{A399908}(i)=\varrho_i(1)$. The first 15 elements of the sequence \oeis{A399908} are
$$1,\, 2,\, 4,\, 11,\, 17,\, 30,\, 40,\, 52,\, 75,\, 102,\,130,\,132,\,166,\,203,\,236.$$ 

The first $100$ terms of sequences \oeis{A399907} and \oeis{A399908} are listed in tables \ref{tabA1} and \ref{tabA2}, respectively. After examining the first $1000$ terms of both sequences, we observed an apparent relation $\oeis{A399908}(i)=\oeis{A399907}(i)+1$ for every $i\geqslant 4$. We state those observations in the following form.

\begin{conjecture}
The following two statements hold.
\begin{enumerate}
\item Every row of $B_\varrho$, except the first one, omits exactly one positive integer, that is, for each $r\geqslant 2$ there exists a unique positive integer $N_r$ such that
$$\mathbb{N}\setminus \{\varrho_r(k): k\geqslant 1\}=\{N_r\}.$$
\item $\oeis{A399908}(i)=\oeis{A399907}(i)+1$ for every $i\geqslant 3$.
\end{enumerate}
\end{conjecture}

The algorithm \ref{alvarrho} in the Appendix generates barrycade $B_\varrho$ and the visualization of the fragment of $B_\varrho$ is presented in figure \ref{Gvarrho}.

\subsection{Grasshopper $\mathbb{N}$-barrycade $B_\psi$}\label{ssBpsi}

Now we will fix the shortcoming of Theorem \ref{infthm1} by building a barrycade in which every row is a permutation of $\mathbb{N}$, proving the following.

\begin{theorem}\label{infthm2}
There exists an $\mathbb{N}$-barrycade of infinite height.
\end{theorem}

\begin{proof}

As in the proof of Theorem \ref{infthm1}, we construct rows in order starting from the first. Having constructed rows form $1$ to $i-1$ and a prefix $P$ of the $i$-th row, we pick $K$ to be the smallest positive integer that has not appeared in the current row. Let $s_P$ be the sum of all terms of $P$. Then we choose the smallest positive integer $K'$ not yet used in the current row, with $K'\neq K$, such that appending $K'$ and then $K$ is admissible; that is, neither of the numbers $s_P+K'$ and $s_P+K'+K$ coincides with a partial sum of any previously constructed row. We then append the two terms $K',K$ to $P$ in this order.

It remains to justify that the auxiliary number $K'$ always exists. As before, each row has at most $O\left(\sqrt{x}\right)$ partial sums smaller than $x$. Since a choice of $K'$ is inadmissible only if $s_P+K'$ or $s_P+K'+K$ is a partial sum of a previously constructed row, it follows that there are at most $O\left(\sqrt{x}\right)$ inadmissible values of $K'$ smaller that $x$, and by considering large enough $x$ we see that it cannot exhaust all positive integers. Therefore, the proof is complete.
\end{proof}

We call the obtained collection of permutations of $\mathbb{N}$ a \emph{grasshopper $\mathbb{N}$-barrycade} $B_\psi=\{\psi_1,\psi_2,\psi_3,\ldots\}$. The first four rows of $B_\psi$ are
\begin{align*}
\psi_1&=(2,\,1,\,4,\,3,\,6,\,5,\,8,\,7,\,10,\,9,\,12,\,11,\,14,\,13,\,16,\,\ldots),\\
\psi_2&=(4,\,1,\,6,\,2,\,7,\,3,\,9,\,5,\,11,\,8,\,13,\,10,\,15,\,12,\,17,\,\ldots),\\
\psi_3&=(8,\,1,\,3,\,2,\,10,\,4,\,6,\,5,\,11,\,7,\,14,\,9,\,15,\,12,\,18,\,\ldots),\\
\psi_4&=(17,\,1,\,7,\,2,\,8,\,3,\,5,\,4,\,11,\,6,\,10,\,9,\,13,\,12,\,18,\,\ldots).
\end{align*}

The algorithm \ref{alpsi} generates barrycade $B_\psi$ and the visualization of the fragment of $B_\psi$ is presented in figure \ref{Gpsi}. 

Let us consider a sequence \oeis{A399897} such that for $i\geqslant1$ we have $\oeis{A399897}(i)$ is the first number in the permutation $\psi_i$. Therefore, the first $15$ terms of \oeis{A399897} are
$$2,\,4,\,8,\,17,\,30,\,44,\,75,\,86,\,116,\,133,\,145,\,204,\,208,\,271,\,294.$$
The first 100 terms of sequence \oeis{A399897} are listed in table \ref{tabA3}. 

Let us observe that the partial sum $1$ does not occur in any row of the barrycade $B_\psi$. Indeed, the construction of each row $\psi_i$ begins by searching for a number $K'$ that allows us to place the number $1$ immediately after $K'$. As a consequence, the sequence \oeis{A399897} is increasing, and hence any partial sum skipped by the first ``jump'' of the grasshopper in a given row is missed forever. The partial sums missed by the barrycade $B_\psi$ in Figure~\ref{Gpsi} are indicated by red dashed vertical lines.

The first $15$ terms of the sequence \oeis{A399898} of partial sums missed by the barrycade $B_\psi$ are
$$1,\,6,\,15,\,19,\,22,\,26,\,33,\,41,\,54,\,59,\,61,\,65,\,68,\,70,\,73$$
and the first $100$ terms are listed in Table~\ref{tabA4}. We conjecture that there are infinitely many more missing partial sums.

\begin{conjecture}
The sequence \oeis{A399898} is infinite.
\end{conjecture}

\begin{figure}[t]
\centering
\begin{tikzpicture}[scale=0.7,barry/.style={line width=1.8pt, line cap=butt},cont/.style={line width=1.8pt, line cap=butt, dash pattern=on 4pt off 3pt}]
\def\W{22}
\def\H{4}

\draw[step=1, gray!60, dash pattern=on 2pt off 2pt] (-0.5,-0.5) grid (\W+0.5,\H+0.5);

\draw[barry] (0,0) -- (\W+0.5,0);
\draw[barry] (0,0) -- (0,\H+0.5);

\draw[cont] (0,\H+0.5) -- (\W+0.5,\H+0.5);
\draw[cont] (\W+0.5,0) -- (\W+0.5,\H+0.5);

\foreach \y in {1,2,3,4} {\draw[barry] (0,\y) -- (\W+0.5,\y);}

\foreach \x in {2,3,7,10,16,21} {\draw[barry] (\x,0) -- (\x,1);}
\foreach \x in {4,5,11,13,20} {\draw[barry] (\x,1) -- (\x,2);}
\foreach \x in {8,9,12,14} {\draw[barry] (\x,2) -- (\x,3);}
\foreach \x in {17,18} {\draw[barry] (\x,3) -- (\x,4);}
\foreach \x in {1,6,15,19,22} {\draw[barry, red, dash pattern=on 3pt off 2pt] (\x,-.5) -- (\x,\H+.5);}
\end{tikzpicture}
\caption{A fragment of the grasshopper $\mathbb{N}$-barrycade $B_\psi$.}
\label{Gpsi}
\end{figure}
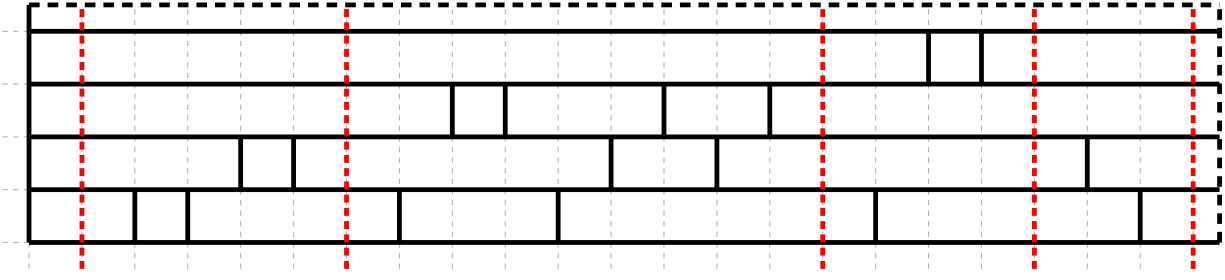

\subsection{Precise grasshopper $\mathbb{N}$-barrycade $B_\varphi$}\label{ssBvarphi}

We are now ready to prove an infinite variant of Conjecture \ref{Guyc} by combining the two previous constructions.

\begin{theorem}\label{infthm3}
There exists a break-free $\mathbb{N}$-barrycade of infinite height.
\end{theorem}
\begin{proof}
Once again, we construct rows in order starting from the first. Having constructed rows form $1$ to $i-1$, we start the $i$-th row with the smallest positive integer $m$ that is not a partial sum of any previous row, therefore ensuring the break-free property. Then, having constructed a prefix $P$ of the $i$-th row, with the sum $s_P$, we choose the smallest positive integer $K'$ not yet used in the current row, with $K'\neq K$, such that neither of the numbers $s_P+K'$ and $s_P+K'+K$ coincides with a partial sum of any earlier row. We then append the two terms $K',K$ to $P$ in this order; this ensures that each row will be a permutation of $\mathbb{N}$.

The proof is complete by observing that at each step the choice of $m$ or $K'$ is possible, which follows by the same argument as in the proofs of Theorems \ref{infthm1} and \ref{infthm2}.
\end{proof}

We call the obtained collection of permutations of $\mathbb{N}$ a \emph{precise grasshopper $\mathbb{N}$-barrycade} $B_\varphi=\{\varphi_1,\varphi_2,\varphi_3,\ldots\}$. The first four rows of $B_\varphi$ are
\begin{align*}
\varphi_1&=(1,\,3,\,2,\,5,\,4,\,7,\,6,\,9,\,8,\,11,\,10,\,13,\,12,\,15,\,14,\,\ldots),\\
\varphi_2&=(2,\,5,\,1,\,6,\,3,\,8,\,4,\,10,\,7,\,12,\,9,\,14,\,11,\,16,\,13,\,\ldots),\\
\varphi_3&=(3,\,6,\,1,\,8,\,2,\,7,\,4,\,11,\,5,\,12,\,9,\,15,\,10,\,16,\,13,\,\ldots),\\
\varphi_4&=(5,\,7,\,1,\,6,\,2,\,9,\,3,\,11,\,4,\,13,\,8,\,15,\,10,\,17,\,12,\,\ldots).
\end{align*}

The algorithm \ref{alvarphi} generate barrycade $G_\varphi$ and the visualization of the fragment of $G_\varphi$ is presented in figure \ref{Gvarphi}. 

One may wonder, to what extent does Theorem \ref{infthm3} support the original version of Conjecture \ref{Guyc}. The construction of $B_\varphi$ ``reaches to infinity'' when choosing subsequent terms of the sequences, which doesn't work in the finite setting, but still it may be interesting to see, how far does it need to reach.

Let us consider a sequence \oeis{A399900} such that for $i\geqslant1$ we have $\oeis{A399900}(i)$ is the first number in the permutation $\varphi_i$. Therefore, the first $15$ terms of \oeis{A399900} are
$$1,\,2,\,3,\,5,\,16,\,26,\,40,\,43,\,54,\,63,\,73,\,78,\,80,\,95,\,100.$$
The first 100 terms of sequence \oeis{A399900} are listed in table \ref{tabA5}. 

\begin{figure}[t]
\centering
\begin{tikzpicture}[scale=0.7,barry/.style={line width=1.8pt, line cap=butt},cont/.style={line width=1.8pt, line cap=butt, dash pattern=on 4pt off 3pt}]
\def\W{22}
\def\H{5}

\draw[step=1, gray!60, dash pattern=on 2pt off 2pt] (-0.5,-0.5) grid (\W+0.5,\H+0.5);

\draw[barry] (0,0) -- (\W+0.5,0);
\draw[barry] (0,0) -- (0,\H+0.5);

\draw[cont] (0,\H+0.5) -- (\W+0.5,\H+0.5);
\draw[cont] (\W+0.5,0) -- (\W+0.5,\H+0.5);

\foreach \y in {1,2,3,4,5} {\draw[barry] (0,\y) -- (\W+0.5,\y);}

\foreach \x in {1,4,6,11,15,22} {\draw[barry] (\x,0) -- (\x,1);}
\foreach \x in {2,7,8,14,17} {\draw[barry] (\x,1) -- (\x,2);}
\foreach \x in {3,9,10,18,20} {\draw[barry] (\x,2) -- (\x,3);}
\foreach \x in {5,12,13,19,21} {\draw[barry] (\x,3) -- (\x,4);}
\foreach \x in {16} {\draw[barry] (\x,4) -- (\x,5);}
\end{tikzpicture}
\caption{A fragment of the precise grasshopper $\mathbb{N}$-barrycade $B_\varphi$.}
\label{Gvarphi}
\end{figure}
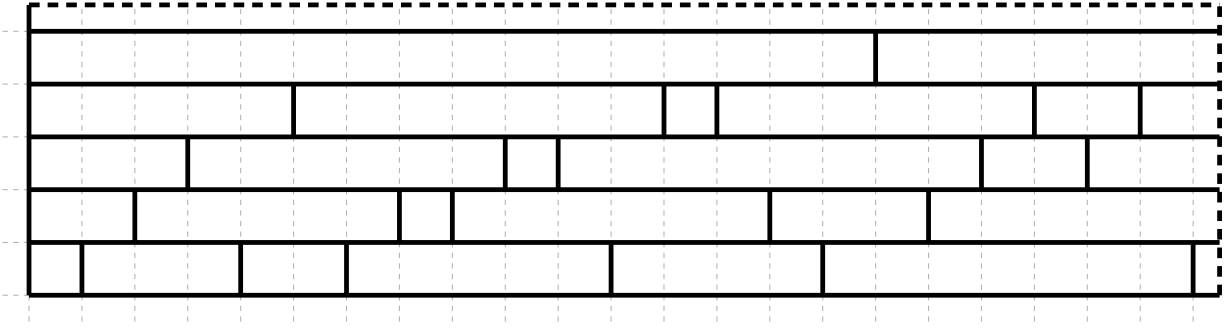

Both grasshopper $B_\psi$ from subsection \ref{ssBpsi} and precise grasshopper $B_\varphi$ from subsection~\ref{ssBvarphi} can be considered in a more ,,greedy'' variant. In this version, we first try to insert the number~$K$; only if this is not possible do we search for a number $K'$ that enables the insertion of $K$. We present these two variants in subsections \ref{ssGpsi} and \ref{ssGvarphi}, respectively.

\subsection{Greedy grasshopper $\mathbb{N}$-barrycade $G_\psi$}\label{ssGpsi}

We now refine the grasshopper construction of barrycade $B_\psi$ by adding a greedy step: before making a jump of the form $K',K$, we first try to append the currently smallest unused number $K$ directly.

In each row, starting with the first, we proceed greedily as follows. At any stage, let $K$ be the smallest positive integer that has not yet appeared in the current row. If appending $K$ does not create a partial sum coinciding with a partial sum of any of the previously constructed rows, then we append $K$. Otherwise, we choose the smallest positive integer $K'$ not yet used in the current row, with $K'\neq K$, such that appending $K'$ and then $K$ is admissible; that is, neither of the two new partial sums coincides with a partial sum of any previously constructed row. We then append the two terms $K',K$ in this order.

Thus the construction remains greedy in two senses: it always tries to insert the smallest unused number $K$, and only when this is impossible does it perform the shortest admissible grasshopper jump that allows $K$ to be inserted immediately afterwards. In this way, every row of $G_\psi$ is forced to be a permutation of $\mathbb{N}$.

We call the obtained collection of permutations of $\mathbb{N}$ a \emph{greedy grasshopper $\mathbb{N}$-barrycade} $G_\psi=\{\psi_1,\psi_2,\psi_3,\ldots\}$. The first four rows of $G_\psi$ are
\begin{align*}
\psi_1&=(1,\,2,\,3,\,4,\,5,\,6,\,7,\,8,\,9,\,10,\,11,\,12,\,13,\,14,\,15,\,\ldots),\\
\psi_2&=(4,\,1,\,2,\,6,\,3,\,8,\,5,\,10,\,7,\,12,\,9,\,14,\,11,\,16,\,13,\,\ldots),\\
\psi_3&=(8,\,1,\,2,\,3,\,4,\,5,\,9,\,6,\,11,\,7,\,13,\,10,\,15,\,12,\,17,\,\ldots),\\
\psi_4&=(19,\,1,\,2,\,3,\,5,\,4,\,6,\,7,\,10,\,8,\,9,\,11,\,12,\,13,\,14,\,\ldots).
\end{align*}

The algorithm \ref{alGpsi} generates barrycade $G_\psi$ and the visualization of the fragment of $G_\psi$ is presented in figure \ref{GGpsi}. Let us consider a sequence \oeis{A399902} such that for $i\geqslant1$ we have $\oeis{A399902}(i)$ is the first number in the permutation $\psi_i$. Therefore, the first $15$ terms of \oeis{A399902} are
$$1,\,4,\,8,\,19,\,26,\,42,\,51,\,72,\,102,\,134,\,143,\,147,\,157,\,168,\,200.$$
The first 100 terms of sequence \oeis{A399902} are listed in table \ref{tabA6}. 

\begin{figure}[t]
\centering
\begin{tikzpicture}[scale=0.7,barry/.style={line width=1.8pt, line cap=butt},cont/.style={line width=1.8pt, line cap=butt, dash pattern=on 4pt off 3pt}]
\def\W{22}
\def\H{4}

\draw[step=1, gray!60, dash pattern=on 2pt off 2pt] (-0.5,-0.5) grid (\W+0.5,\H+0.5);

\draw[barry] (0,0) -- (\W+0.5,0);
\draw[barry] (0,0) -- (0,\H+0.5);

\draw[cont] (0,\H+0.5) -- (\W+0.5,\H+0.5);
\draw[cont] (\W+0.5,0) -- (\W+0.5,\H+0.5);

\foreach \y in {1,2,3,4} {\draw[barry] (0,\y) -- (\W+0.5,\y);}

\foreach \x in {1,3,6,10,15,21} {\draw[barry] (\x,0) -- (\x,1);}
\foreach \x in {4,5,7,13,16} {\draw[barry] (\x,1) -- (\x,2);}
\foreach \x in {8,9,11,14,18} {\draw[barry] (\x,2) -- (\x,3);}
\foreach \x in {19,20,22} {\draw[barry] (\x,3) -- (\x,4);}
\foreach \x in {2,12,17} {\draw[barry, red, dash pattern=on 3pt off 2pt] (\x,-.5) -- (\x,\H+.5);}
\end{tikzpicture}
\caption{A fragment of the greedy grashopper $\mathbb{N}$-barrycade $G_\psi$.}
\label{GGpsi}
\end{figure}

As in the grasshopper barrycade $B_\psi$, some partial sums never occur in $G_\psi$. The first $15$ terms of the sequence \oeis{A399903} of partial sums missed by the barrycade $G_\psi$ are
$$2,\,12,\,17,\,35,\,37,\,62,\,77,\,84,\,90,\,93,\,100,\,113,\,118,\,129,\,151$$
and the first $100$ terms are listed in Table~\ref{tabA7}. We state the following

\begin{conjecture}
The sequence \oeis{A399903} is infinite.
\end{conjecture}

\subsection{Precise greedy grasshopper $\mathbb{N}$-barrycade $G_\varphi$}\label{ssGvarphi}

We now introduce a greedy modification of the precise grasshopper $\mathbb{N}$-barrycade $B_\varphi$. As before, each row starts by jumping to the first available free position. After this initial step, however, we first try to append the smallest unused number directly; only if this is impossible do we use a grasshopper jump.

In each row, starting with the first, we proceed as follows. First, we jump to the first available free position. More precisely, if $m$ is the smallest positive integer not belonging to $P_{r-1}$, then we choose the first term of the row to be $m$. After this initial jump, the row is continued greedily.

At any later stage, let $K$ be the smallest positive integer that has not yet appeared in the current row. If appending $K$ does not create a partial sum coinciding with a partial sum of any previously constructed row, then we append $K$. Otherwise, we choose the smallest positive integer $K'$ not yet used in the current row, with $K'\neq K$, such that appending $K'$ and then $K$ is admissible; that is, neither of the two new partial sums coincides with a partial sum of any previously constructed row. We then append the two terms $K',K$ in this order.

We call the obtained collection of permutations of $\mathbb{N}$ a \emph{precise greedy grasshopper $\mathbb{N}$-barrycade} $G_\varphi=\{\varphi_1,\varphi_2,\varphi_3,\ldots\}$. The first four rows of $G_\varphi$ are
\begin{align*}
\varphi_1&=(1,\,2,\,3,\,4,\,5,\,6,\,7,\,8,\,9,\,10,\,11,\,12,\,13,\,14,\,15\,\ldots),\\
\varphi_2&=(2,\,5,\,1,\,3,\,7,\,4,\,9,\,6,\,11,\,8,\,13,\,10,\,15,\,12,\,17,\,\ldots),\\
\varphi_3&=(4,\,1,\,7,\,2,\,3,\,8,\,5,\,10,\,6,\,12,\,9,\,14,\,11,\,16,\,13,\,\ldots),\\
\varphi_4&=(9,\,10,\,1,\,4,\,2,\,3,\,5,\,7,\,6,\,13,\,8,\,14,\,11,\,17,\,12,\,\ldots).
\end{align*}

The algorithm \ref{alGvarphi} generates barrycade $G_\varphi$ and the visualization of the fragment of $G_\varphi$ is presented in figure \ref{GGvarphi}. Let us consider a sequence \oeis{A399905} such that for $i\geqslant1$ we have $\oeis{A399905}(i)$ is the first number in the permutation $\varphi_i$. Therefore, the first $15$ terms of \oeis{A399905} are
$$1,\,2,\,4,\,9,\,13,\,16,\,23,\,27,\,39,\,51,\,53,\,57,\,61,\,63,\,65.$$
The first 100 terms of sequence \oeis{A399905} are listed in table \ref{tabA8}. 

\begin{figure}[t]
\centering
\begin{tikzpicture}[scale=0.7,barry/.style={line width=1.8pt, line cap=butt},cont/.style={line width=1.8pt, line cap=butt, dash pattern=on 4pt off 3pt}]
\def\W{22}
\def\H{5}

\draw[step=1, gray!60, dash pattern=on 2pt off 2pt] (-0.5,-0.5) grid (\W+0.5,\H+0.5);

\draw[barry] (0,0) -- (\W+0.5,0);
\draw[barry] (0,0) -- (0,\H+0.5);

\draw[cont] (0,\H+0.5) -- (\W+0.5,\H+0.5);
\draw[cont] (\W+0.5,0) -- (\W+0.5,\H+0.5);

\foreach \y in {1,2,3,4,5} {\draw[barry] (0,\y) -- (\W+0.5,\y);}

\foreach \x in {1,3,6,10,15,21} {\draw[barry] (\x,0) -- (\x,1);}
\foreach \x in {2, 7, 8, 11, 18, 22} {\draw[barry] (\x,1) -- (\x,2);}
\foreach \x in {4, 5, 12, 14, 17} {\draw[barry] (\x,2) -- (\x,3);}
\foreach \x in {9, 19, 20} {\draw[barry] (\x,3) -- (\x,4);}
\foreach \x in {13} {\draw[barry] (\x,4) -- (\x,5);}
\end{tikzpicture}
\caption{A fragment of the precise greedy grashopper $\mathbb{N}$-barrycade $G_\varphi$.}
\label{GGvarphi}
\end{figure}

\section{Final thoughts}

Let us note that, apart from the graphical representation, a barrycade can also be encoded as a word. Let us consider a barrycade $B=(\pi_1,\pi_2,\ldots,\pi_h)$ of size~$n$ and height $h$ and let $s=\frac{n(n+1)}2-1$. Barrycade $B$ determines a word
$$w_B=b_1b_2\ldots b_s$$
over the alphabet $\{\mathtt0,\mathtt1,\ldots,\mathtt{h}\}$, where, for each $k\in\left[s\right]$,
$$b_k=
\begin{cases}
\mathtt{j}, & \text{if } k\in S_{\pi_j},\\
\mathtt{0}, & \text{if } k\notin \bigcup_{i=1}^h S_{\pi_i}.
\end{cases}
$$
This is well-defined, since the sets $S_{\pi_1},S_{\pi_2},\ldots,S_{\pi_h}$ are pairwise disjoint. Let us notice that every break-free barrycade can be considered over a smaller alphabet, namely $\{\mathtt1,\mathtt2,\ldots,\mathtt{h}\}$.

For the five infinite constructions from Section~3, the corresponding word representations are recorded in the OEIS as \oeis{A399909} for $B_\varrho$, \oeis{A399899} for $B_\psi$, \oeis{A399901} for $B_\varphi$, \oeis{A399904} for $G_\psi$, and \oeis{A399906} for $G_\varphi$ \cite{OEIS}.

Let us notice that the barrycade presented in figure \ref{Barry4} can be decoded by a word
$$\mathtt{12332414231413224134}.$$

This encoding admits the following simple characterization. Let $w=b_1b_2\ldots b_{n(n+1)/2}$ be a~word over the alphabet $\{\mathtt0,\mathtt1,\ldots,\mathtt{h}\}$. For a letter $\mathtt{j}$, where $j\in[h]$, let $$1\leqslant p_1^{(j)}<p_2^{(j)}<\ldots<p_{n-1}^{(j)}\leqslant b(b+1)/2-1$$ be the positions at which $\mathtt{j}$ occurs in $w$. Then $w$ encodes an $n$-barrycade if and only if, for every $j\in[h]$, the numbers
$$p_1^{(j)},\quad p_2^{(j)}-p_1^{(j)},\quad p_3^{(j)}-p_2^{(j)},\quad \ldots,\quad p_{n-1}^{(j)}-p_{n-2}^{(j)},\quad \frac{n(n+1)}2-p_{n-1}^{(j)}$$
are precisely the elements of $[n]$, each appearing exactly once.

Indeed, these numbers are the lengths of the consecutive logs in the $j$-th row of the barrycade. Thus the distances between consecutive occurrences of the same non-zero letter, together with the distances from the first and last occurrence to the two ends of the word, form a permutation of $[n]$.

In particular, a break-free $n$-barrycade is exactly such a word with no occurrences of $\mathtt0$. Equivalently, it is a word $$w=b_1b_2\ldots b_{T-1}$$ over the alphabet $\{\mathtt1,\mathtt2,\ldots,\mathtt{h}\}$ such that, for each $j\in[h]$, the letter $\mathtt{j}$ occurs exactly $(n-1)$ times and the $n$ distances determined by its occurrences and by the two ends of the word are precisely $1,2,\ldots,n$.

Finally, let us make an obvious remark that the barrycade problem can be studied for arbitrary given sequence of integers. To formulate it precisely, let $A=(a_1,a_2,\ldots)$ be any increasing sequence of positive integers. An $n$-barrycade of height $h$ is just a collection of $h$ permutations of the prefix $(a_1,a_2,\ldots,a_n)$ of $A$ whose all proper partial sums are pairwise different. Let $h_A(n)$ be the maximum height of an $n$-barrycade over $A$. It would be nice to know how this function behaves in case of some famous sequences, like primes or Fibonacci numbers.

\appendix

\section{Algorithms and tables}\label{apA}

\begin{algorithm}[H]
\caption{Construction of a greedy $\mathbb{N}$-quasi-barrycade $B_\varrho$}\label{alvarrho}
\begin{algorithmic}[1]
\State $P_0\gets \emptyset$
\For{$r=1,2,3,\ldots$}
    \State $\varrho_r\gets ()$
    \State $s\gets 0$
    \State $U\gets \emptyset$
    \For{$k=1,2,3,\ldots$}
        \State Choose the smallest positive integer $a$ such that
        \[
            a\notin U
            \quad\text{and}\quad
            s+a\notin P_{r-1}.
        \]
        \State Append $a$ to $\varrho_r$
        \State $U\gets U\cup\{a\}$
        \State $s\gets s+a$
    \EndFor
    \State $P_r\gets P_{r-1}\cup S_{\varrho_r}$
\EndFor
\end{algorithmic}
\end{algorithm}

\begin{table}
\centering
\begin{tabular}{|c|c|c|c|c|c|c|c|c|c|c|}
\hline
$\varrho_2-\varrho_{11}$ & $1$ & $2$ & $10$ & $16$ & $29$ & $39$ & $51$ & $74$ & $101$ & $129$ \\\hline
$\varrho_{12}-\varrho_{21}$ & $131$ & $165$ & $202$ & $235$ & $281$ & $320$ & $369$ & $393$ & $443$ & $505$ \\\hline
$\varrho_{22}-\varrho_{31}$ & $549$ & $583$ & $644$ & $685$ & $776$ & $840$ & $925$ & $967$ & $1044$ & $1107$ \\\hline
$\varrho_{32}-\varrho_{41}$ & $1198$ & $1247$ & $1345$ & $1389$ & $1522$ & $1568$ & $1693$ & $1799$ & $1857$ & $1997$ \\\hline
$\varrho_{42}-\varrho_{51}$ & $2097$ & $2196$ & $2272$ & $2394$ & $2506$ & $2510$ & $2788$ & $2845$ & $2961$ & $3046$ \\\hline
$\varrho_{52}-\varrho_{61}$ & $3118$ & $3309$ & $3316$ & $3564$ & $3734$ & $3937$ & $4041$ & $4200$ & $4312$ & $4441$ \\\hline
$\varrho_{62}-\varrho_{71}$ & $4545$ & $4778$ & $4905$ & $5013$ & $5107$ & $5381$ & $5507$ & $5736$ & $5868$ & $5936$ \\\hline
$\varrho_{72}-\varrho_{81}$ & $6271$ & $6383$ & $6591$ & $6826$ & $6939$ & $7120$ & $7198$ & $7328$ & $7491$ & $7772$ \\\hline
$\varrho_{82}-\varrho_{91}$ & $8039$ & $8340$ & $8344$ & $8698$ & $8863$ & $9168$ & $9373$ & $9483$ & $9725$ & $9922$ \\\hline
$\varrho_{92}-\varrho_{101}$ & $10\,137$ & $10\,493$ & $10\,627$ & $10\,818$ & $11\,095$ & $11\,373$ & $11\,541$ & $11\,750$ & $12\,019$ & $12\,340$ \\
\hline
\end{tabular}
\vskip0.5cm
\caption{The first $100$ elements of the sequence \oeis{A399907} of the smallest numbers omitted in the rows $\varrho_r$ for $r\geqslant 2$.}
\label{tabA1}
\end{table}

\bigskip

\begin{table}
\centering
\begin{tabular}{|c|c|c|c|c|c|c|c|c|c|c|}
\hline
$\varrho_1-\varrho_{10}$ & $1$ & $2$ & $4$ & $11$ & $17$ & $30$ & $40$ & $52$ & $75$ & $102$ \\\hline
$\varrho_{11}-\varrho_{20}$ & $130$ & $132$ & $166$ & $203$ & $236$ & $282$ & $321$ & $370$ & $394$ & $444$ \\\hline
$\varrho_{21}-\varrho_{30}$ & $506$ & $550$ & $584$ & $645$ & $686$ & $777$ & $841$ & $926$ & $968$ & $1045$ \\\hline
$\varrho_{31}-\varrho_{40}$ & $1108$ & $1199$ & $1248$ & $1346$ & $1390$ & $1523$ & $1569$ & $1694$ & $1800$ & $1858$ \\\hline
$\varrho_{41}-\varrho_{50}$ & $1998$ & $2098$ & $2197$ & $2273$ & $2395$ & $2507$ & $2511$ & $2789$ & $2846$ & $2962$ \\\hline
$\varrho_{51}-\varrho_{60}$ & $3047$ & $3119$ & $3310$ & $3317$ & $3565$ & $3735$ & $3938$ & $4042$ & $4201$ & $4313$ \\\hline
$\varrho_{61}-\varrho_{70}$ & $4442$ & $4546$ & $4779$ & $4906$ & $5014$ & $5108$ & $5382$ & $5508$ & $5737$ & $5869$ \\\hline
$\varrho_{71}-\varrho_{80}$ & $5937$ & $6272$ & $6384$ & $6592$ & $6827$ & $6940$ & $7121$ & $7199$ & $7329$ & $7492$ \\\hline
$\varrho_{81}-\varrho_{90}$ & $7773$ & $8040$ & $8341$ & $8345$ & $8699$ & $8864$ & $9169$ & $9374$ & $9484$ & $9726$ \\\hline
$\varrho_{91}-\varrho_{100}$ & $9923$ & $10\,138$ & $10\,494$ & $10\,628$ & $10\,819$ & $11\,096$ & $11\,374$ & $11\,542$ & $11\,751$ & $12\,020$ \\
\hline
\end{tabular}
\vskip0.5cm
\caption{The first $100$ elements of the sequence \oeis{A399908} of the first elements in rows $\varrho_r$ for $r\geqslant 1$.}
\label{tabA2}
\end{table}

\begin{table}
\centering
\begin{tabular}{|c|c|c|c|c|c|c|c|c|c|c|}
\hline
$\psi_1-\psi_{10}$ & $2$ & $4$ & $8$ & $17$ & $30$ & $44$ & $75$ & $86$ & $116$ & $133$ \\\hline
$\psi_{11}-\psi_{20}$ & $145$ & $204$ & $208$ & $271$ & $294$ & $322$ & $345$ & $364$ & $368$ & $404$ \\\hline
$\psi_{21}-\psi_{30}$ & $426$ & $482$ & $506$ & $557$ & $602$ & $853$ & $939$ & $979$ & $1059$ & $1109$ \\\hline
$\psi_{31}-\psi_{40}$ & $1117$ & $1124$ & $1148$ & $1219$ & $1284$ & $1309$ & $1368$ & $1511$ & $1585$ & $1872$ \\\hline
$\psi_{41}-\psi_{50}$ & $2032$ & $2134$ & $2193$ & $2324$ & $2398$ & $2549$ & $2729$ & $2877$ & $2906$ & $3050$ \\\hline
$\psi_{51}-\psi_{60}$ & $3315$ & $3384$ & $3547$ & $3817$ & $3854$ & $3880$ & $3908$ & $4058$ & $4145$ & $4164$ \\\hline
$\psi_{61}-\psi_{70}$ & $4461$ & $4647$ & $4786$ & $4807$ & $5193$ & $5234$ & $5523$ & $5656$ & $5833$ & $5843$ \\\hline
$\psi_{71}-\psi_{80}$ & $5934$ & $6099$ & $6212$ & $6262$ & $6396$ & $6604$ & $6756$ & $6890$ & $7069$ & $7102$ \\\hline
$\psi_{81}-\psi_{90}$ & $7553$ & $7581$ & $7671$ & $8057$ & $8171$ & $8203$ & $9222$ & $9266$ & $9278$ & $9640$ \\\hline
$\psi_{91}-\psi_{100}$ & $9714$ & $9792$ & $9996$ & $10\,393$ & $10\,583$ & $10\,717$ & $10\,989$ & $12\,081$ & $12\,172$ & $12\,382$ \\
\hline
\end{tabular}
\vskip0.5cm
\caption{The first $100$ elements of the sequence \oeis{A399897} of the first elements in rows $\psi_r$ of the barrycade $B_\psi$ for $r\geqslant 1$.}
\label{tabA3}
\end{table}

\begin{algorithm}[H]
\caption{Construction of a grasshopper $\mathbb{N}$-barrycade $B_\psi$}\label{alpsi}
\begin{algorithmic}[1]
\State $P_0\gets \emptyset$
\For{$r=1,2,3,\ldots$}
    \State $\psi_r\gets ()$
    \State $s\gets 0$
    \State $U\gets \emptyset$
    \For{$k=1,2,3,\ldots$}
        \State Let $K$ be the smallest positive integer such that $K\notin U$
        \State Choose the smallest positive integer $K'$ such that
        \[
            K'\notin U,\qquad K'\neq K,\qquad
            s+K'\notin P_{r-1},\qquad
            s+K'+K\notin P_{r-1}.
        \]
        \State Append $K'$ to $\psi_r$
        \State Append $K$ to $\psi_r$
        \State $U\gets U\cup\{K',K\}$
        \State $s\gets s+K'+K$
    \EndFor
    \State $P_r\gets P_{r-1}\cup S_{\psi_r}$
\EndFor
\end{algorithmic}
\end{algorithm}

\begin{table}
\centering
\begin{tabular}{|c|c|c|c|c|c|c|c|c|c|c|}
\hline
$a_1-a_{10}$ & $1$ & $6$ & $15$ & $19$ & $22$ & $26$ & $33$ & $41$ & $54$ & $59$ \\\hline
$a_{11}-a_{20}$ & $61$ & $65$ & $68$ & $70$ & $73$ & $90$ & $101$ & $103$ & $110$ & $114$ \\\hline
$a_{21}-a_{30}$ & $119$ & $130$ & $135$ & $150$ & $153$ & $163$ & $170$ & $179$ & $185$ & $192$ \\\hline
$a_{31}-a_{40}$ & $198$ & $223$ & $240$ & $242$ & $256$ & $268$ & $274$ & $291$ & $299$ & $310$ \\\hline
$a_{41}-a_{50}$ & $315$ & $318$ & $338$ & $361$ & $378$ & $385$ & $389$ & $393$ & $399$ & $409$ \\\hline
$a_{51}-a_{60}$ & $420$ & $437$ & $452$ & $477$ & $487$ & $491$ & $496$ & $500$ & $517$ & $523$ \\\hline
$a_{61}-a_{70}$ & $549$ & $572$ & $578$ & $589$ & $621$ & $627$ & $665$ & $673$ & $676$ & $689$ \\\hline
$a_{71}-a_{80}$ & $695$ & $698$ & $703$ & $721$ & $736$ & $759$ & $768$ & $777$ & $779$ & $799$ \\\hline
$a_{81}-a_{90}$ & $808$ & $810$ & $814$ & $830$ & $846$ & $848$ & $855$ & $860$ & $893$ & $895$ \\\hline
$a_{91}-a_{100}$ & $900$ & $902$ & $922$ & $924$ & $941$ & $945$ & $966$ & $970$ & $985$ & $988$ \\
\hline
\end{tabular}
\vskip0.5cm
\caption{The first $100$ elements of the sequence \oeis{A399898} of partial sums missing from the barrycade $B_\psi$.}
\label{tabA4}
\end{table}

\begin{algorithm}[H]
\caption{Construction of a precise grasshopper $\mathbb{N}$-barrycade $B_\varphi$}\label{alvarphi}
\begin{algorithmic}[1]
\State $P_0\gets \emptyset$
\For{$r=1,2,3,\ldots$}
    \State $\varphi_r\gets ()$
    \State $U\gets \emptyset$
    \State Let $m$ be the smallest positive integer such that $m\notin P_{r-1}$
    \State Append $m$ to $\varphi_r$
    \State $U\gets U\cup\{m\}$
    \State $s\gets m$
    \For{$k=1,2,3,\ldots$}
        \State Let $K$ be the smallest positive integer such that $K\notin U$
        \State Choose the smallest positive integer $K'$ such that
        \[
            K'\notin U,\qquad K'\neq K,\qquad
            s+K'\notin P_{r-1},\qquad
            s+K'+K\notin P_{r-1}.
        \]
        \State Append $K'$ to $\varphi_r$
        \State Append $K$ to $\varphi_r$
        \State $U\gets U\cup\{K',K\}$
        \State $s\gets s+K'+K$
    \EndFor
    \State $P_r\gets P_{r-1}\cup S_{\varphi_r}$
\EndFor
\end{algorithmic}
\end{algorithm}

\begin{table}
\centering
\begin{tabular}{|c|c|c|c|c|c|c|c|c|c|c|}
\hline
$\varphi_1-\varphi_{10}$ & $1$ & $2$ & $3$ & $5$ & $16$ & $26$ & $40$ & $43$ & $54$ & $63$ \\\hline
$\varphi_{11}-\varphi_{20}$ & $73$ & $78$ & $80$ & $95$ & $100$ & $133$ & $136$ & $143$ & $151$ & $159$ \\\hline
$\varphi_{21}-\varphi_{30}$ & $166$ & $168$ & $171$ & $173$ & $200$ & $204$ & $209$ & $268$ & $275$ & $287$ \\\hline
$\varphi_{31}-\varphi_{40}$ & $295$ & $298$ & $313$ & $323$ & $337$ & $343$ & $349$ & $365$ & $387$ & $390$ \\\hline
$\varphi_{41}-\varphi_{50}$ & $406$ & $408$ & $425$ & $451$ & $510$ & $526$ & $544$ & $554$ & $587$ & $590$ \\\hline
$\varphi_{51}-\varphi_{60}$ & $621$ & $624$ & $627$ & $640$ & $646$ & $655$ & $661$ & $695$ & $698$ & $710$ \\\hline
$\varphi_{61}-\varphi_{70}$ & $755$ & $765$ & $775$ & $797$ & $812$ & $815$ & $828$ & $850$ & $853$ & $855$ \\\hline
$\varphi_{71}-\varphi_{80}$ & $860$ & $872$ & $877$ & $884$ & $889$ & $894$ & $902$ & $920$ & $924$ & $939$ \\\hline
$\varphi_{81}-\varphi_{90}$ & $960$ & $963$ & $970$ & $990$ & $1012$ & $1016$ & $1020$ & $1029$ & $1051$ & $1061$ \\\hline
$\varphi_{91}-\varphi_{100}$ & $1065$ & $1076$ & $1105$ & $1107$ & $1124$ & $1168$ & $1178$ & $1202$ & $1218$ & $1254$ \\
\hline
\end{tabular}
\vskip0.5cm
\caption{The first $100$ elements of the sequence \oeis{A399900} of the first elements in rows $\varphi_r$ of the precise grasshopper barrycade $B_\varphi$ for $r\geqslant 1$.}
\label{tabA5}
\end{table}

\begin{algorithm}[H]
\caption{Construction of a greedy grasshopper $\mathbb{N}$-barrycade $G_\psi$}\label{alGpsi}
\begin{algorithmic}[1]
\State $P_0\gets \emptyset$
\For{$r=1,2,3,\ldots$}
    \State $\psi_r\gets ()$
    \State $s\gets 0$
    \State $U\gets \emptyset$
    \For{$k=1,2,3,\ldots$}
        \State Let $K$ be the smallest positive integer such that $K\notin U$
        \If{$s+K\notin P_{r-1}$}
            \State Append $K$ to $\psi_r$
            \State $U\gets U\cup\{K\}$
            \State $s\gets s+K$
        \Else
            \State Choose the smallest positive integer $K'$ such that
            \[
                K'\notin U,\qquad K'\neq K,\qquad
                s+K'\notin P_{r-1},\qquad
                s+K'+K\notin P_{r-1}.
            \]
            \State Append $K'$ to $\psi_r$
            \State Append $K$ to $\psi_r$
            \State $U\gets U\cup\{K',K\}$
            \State $s\gets s+K'+K$
        \EndIf
    \EndFor
    \State $P_r\gets P_{r-1}\cup S_{\psi_r}$
\EndFor
\end{algorithmic}
\end{algorithm}

\begin{table}
\centering
\begin{tabular}{|c|c|c|c|c|c|c|c|c|c|c|}
\hline
$\psi_1-\psi_{10}$ & $1$ & $4$ & $8$ & $19$ & $26$ & $42$ & $51$ & $72$ & $102$ & $134$ \\\hline
$\psi_{11}-\psi_{20}$ & $143$ & $147$ & $157$ & $168$ & $200$ & $266$ & $291$ & $296$ & $309$ & $314$ \\\hline
$\psi_{21}-\psi_{30}$ & $321$ & $526$ & $545$ & $556$ & $593$ & $699$ & $836$ & $854$ & $900$ & $940$ \\\hline
$\psi_{31}-\psi_{40}$ & $1252$ & $1310$ & $1414$ & $1427$ & $1586$ & $1592$ & $1648$ & $1826$ & $1887$ & $1922$ \\\hline
$\psi_{41}-\psi_{50}$ & $2055$ & $2331$ & $2373$ & $2547$ & $2554$ & $2614$ & $2624$ & $2770$ & $2823$ & $2912$ \\\hline
$\psi_{51}-\psi_{60}$ & $2995$ & $3153$ & $3196$ & $3210$ & $3515$ & $3549$ & $3980$ & $4310$ & $4541$ & $4732$ \\\hline
$\psi_{61}-\psi_{70}$ & $4821$ & $4987$ & $5015$ & $5048$ & $5113$ & $5144$ & $5652$ & $5765$ & $6083$ & $6405$ \\\hline
$\psi_{71}-\psi_{80}$ & $6623$ & $6708$ & $6758$ & $6834$ & $6869$ & $6944$ & $7005$ & $7441$ & $7720$ & $7725$ \\\hline
$\psi_{81}-\psi_{90}$ & $7792$ & $7797$ & $8111$ & $8310$ & $8635$ & $8763$ & $9580$ & $9665$ & $9960$ & $9985$ \\\hline
$\psi_{91}-\psi_{100}$ & $10\,290$ & $10\,437$ & $10\,720$ & $10\,855$ & $10\,957$ & $11\,001$ & $11\,108$ & $11\,157$ & $11\,294$ & $11\,298$ \\
\hline
\end{tabular}
\vskip0.5cm
\caption{The first $100$ elements of the sequence \oeis{A399902} of the first elements in rows $\psi_r$ of the barrycade $G_\psi$ for $r\geqslant 1$.}
\label{tabA6}
\end{table}

\begin{table}
\centering
\begin{tabular}{|c|c|c|c|c|c|c|c|c|c|c|}
\hline
$a_1-a_{10}$ & $2$ & $12$ & $17$ & $35$ & $37$ & $62$ & $77$ & $84$ & $90$ & $93$ \\\hline
$a_{11}-a_{20}$ & $100$ & $113$ & $118$ & $129$ & $151$ & $161$ & $176$ & $183$ & $186$ & $189$ \\\hline
$a_{21}-a_{30}$ & $196$ & $214$ & $227$ & $244$ & $271$ & $334$ & $350$ & $366$ & $373$ & $391$ \\\hline
$a_{31}-a_{40}$ & $399$ & $405$ & $421$ & $428$ & $432$ & $451$ & $460$ & $471$ & $484$ & $494$ \\\hline
$a_{41}-a_{50}$ & $552$ & $566$ & $583$ & $588$ & $610$ & $617$ & $662$ & $682$ & $701$ & $724$ \\\hline
$a_{51}-a_{60}$ & $731$ & $748$ & $760$ & $773$ & $788$ & $800$ & $815$ & $817$ & $846$ & $872$ \\\hline
$a_{61}-a_{70}$ & $896$ & $945$ & $965$ & $977$ & $983$ & $1003$ & $1010$ & $1030$ & $1049$ & $1056$ \\\hline
$a_{71}-a_{80}$ & $1063$ & $1077$ & $1080$ & $1100$ & $1119$ & $1123$ & $1127$ & $1143$ & $1157$ & $1165$ \\\hline
$a_{81}-a_{90}$ & $1169$ & $1174$ & $1207$ & $1219$ & $1221$ & $1224$ & $1237$ & $1243$ & $1254$ & $1269$ \\\hline
$a_{91}-a_{100}$ & $1272$ & $1302$ & $1314$ & $1322$ & $1352$ & $1364$ & $1372$ & $1374$ & $1396$ & $1404$ \\
\hline
\end{tabular}
\vskip0.5cm
\caption{The first $100$ terms of the sequence \oeis{A399903} of partial sums missing from the barrycade $G_\psi$.}
\label{tabA7}
\end{table}

\clearpage
\begin{algorithm}[H]
\caption{Construction of a precise greedy grasshopper $\mathbb{N}$-barrycade $G_\varphi$}\label{alGvarphi}
\begin{algorithmic}[1]
\footnotesize
\State $P_0\gets \emptyset$
\For{$r=1,2,3,\ldots$}
    \State $\varphi_r\gets ()$
    \State $U\gets \emptyset$
    \State Let $m$ be the smallest positive integer such that $m\notin P_{r-1}$
    \State Append $m$ to $\varphi_r$
    \State $U\gets U\cup\{m\}$
    \State $s\gets m$
    \For{$k=1,2,3,\ldots$}
        \State Let $K$ be the smallest positive integer such that $K\notin U$
        \If{$s+K\notin P_{r-1}$}
            \State Append $K$ to $\varphi_r$
            \State $U\gets U\cup\{K\}$
            \State $s\gets s+K$
        \Else
            \State Choose the smallest positive integer $K'$ such that
            \Statex \hspace{\algorithmicindent}$K'\notin U$, $K'\neq K$, $s+K'\notin P_{r-1}$, and
            \Statex \hspace{\algorithmicindent}$s+K'+K\notin P_{r-1}$.
            \State Append $K'$ to $\varphi_r$
            \State Append $K$ to $\varphi_r$
            \State $U\gets U\cup\{K',K\}$
            \State $s\gets s+K'+K$
        \EndIf
    \EndFor
    \State $P_r\gets P_{r-1}\cup S_{\varphi_r}$
\EndFor
\end{algorithmic}
\end{algorithm}

\begin{table}[H]
\centering
\begin{tabular}{|c|c|c|c|c|c|c|c|c|c|c|}
\hline
$a_1-a_{10}$ & $1$ & $2$ & $4$ & $9$ & $13$ & $16$ & $23$ & $27$ & $39$ & $51$ \\\hline
$a_{11}-a_{20}$ & $53$ & $57$ & $61$ & $63$ & $65$ & $74$ & $90$ & $100$ & $102$ & $118$ \\\hline
$a_{21}-a_{30}$ & $142$ & $145$ & $152$ & $159$ & $164$ & $170$ & $178$ & $216$ & $230$ & $242$ \\\hline
$a_{31}-a_{40}$ & $255$ & $274$ & $292$ & $294$ & $299$ & $312$ & $338$ & $342$ & $349$ & $364$ \\\hline
$a_{41}-a_{50}$ & $369$ & $387$ & $394$ & $428$ & $431$ & $451$ & $457$ & $480$ & $483$ & $486$ \\\hline
$a_{51}-a_{60}$ & $489$ & $541$ & $544$ & $547$ & $556$ & $570$ & $574$ & $578$ & $586$ & $591$ \\\hline
$a_{61}-a_{70}$ & $602$ & $625$ & $663$ & $714$ & $721$ & $726$ & $738$ & $756$ & $775$ & $797$ \\\hline
$a_{71}-a_{80}$ & $817$ & $844$ & $855$ & $867$ & $876$ & $883$ & $889$ & $892$ & $900$ & $920$ \\\hline
$a_{81}-a_{90}$ & $933$ & $942$ & $945$ & $966$ & $975$ & $985$ & $987$ & $1025$ & $1028$ & $1059$ \\\hline
$a_{91}-a_{100}$ & $1076$ & $1080$ & $1120$ & $1153$ & $1156$ & $1159$ & $1170$ & $1219$ & $1244$ & $1249$ \\
\hline
\end{tabular}
\vskip0.5cm
\caption{The first $100$ terms of the sequence \oeis{A399905} of the first elements in rows $\varphi_r$ of the barrycade $G_\varphi$.}
\label{tabA8}
\end{table}

Prefixes of length 100 of the word representations of infinite barrycades:
\begin{align*}
B_\varrho\ (\oeis{A399909}):&\,\mathtt{12132133214432145432154553214654632165475632174656}\\
&\,\mathtt{68321475677573217468858793217648598967321899458769}\\
B_\psi\ (\oeis{A399899}):&\,\mathtt{01122013312323014402102340431552034124350546614253}\\
&\,\mathtt{65601234060564051020350477512367476887507162345678}\\
B_\varphi\ (\oeis{A399901}):&\,\mathtt{12314122331442152343415526312435456615275384123465}\\
&\,\mathtt{677951623647(10)751234657(11)6788(12)1(13)2634576997}\\
&\,\mathtt{1234(14)5768(15)}\\
G_\psi\ (\oeis{A399904}):&\,\mathtt{10122123313023120344143245512453540103245665124636}\\
&\,\mathtt{77651342567076412735788475013627504687801203574860}\\
G_\varphi\ (\oeis{A399906}):&\,\mathtt{12133122412353163244127434814325545125934566134256}\\
&\,\mathtt{(10)6(11)512(12)364(13)6(14)5(15)13425677(16)688126347}\\
&\,\mathtt{878568(17)134257867(18)}
\end{align*}

\end{document}